\documentclass{article}

\usepackage[utf8]{inputenc}
\usepackage{amssymb,amsmath,mathdots,amsthm}

\usepackage[a4paper,
    inner=2cm,outer=3cm,bottom=3cm,top=2cm,
    marginparsep=1em,marginparwidth=2cm]{geometry}

\usepackage{enumitem}
\usepackage{color}

\newcommand{\bu}{\boldsymbol{u}}
\newcommand{\Parikh}[1][\kern0.3ex]{\mathrm{P}_{\kern-0.3ex#1}}

\newcommand{\N}{\mathbb{N}}
\newcommand{\Z}{\mathbb{Z}}

\newcommand{\R}{\mathbb{R}}
\newcommand{\A}{\mathcal{A}}
\newcommand{\B}{\mathcal{B}}

\newtheorem{thm}{Theorem}
\newtheorem{prop}[thm]{Proposition}

\newtheorem{lem}[thm]{Lemma}

\theoremstyle{definition}

\newtheorem{de}[thm]{Definition}
\newtheorem{ex}[thm]{Example}

\usepackage{tikz}
\usetikzlibrary{calc}
\usetikzlibrary{intersections}
\usetikzlibrary{decorations.pathreplacing}

\usepackage[
   backend=bibtex,  
   sorting=nyt,       
   giveninits=true,   
   isbn=false         
]{biblatex}

\DeclareFieldFormat
  [article,inbook,incollection,inproceedings,patent,thesis,unpublished]
  {title}{{\textit{#1}\isdot}}

\DeclareFieldFormat{journaltitle}{#1}
\DeclareFieldFormat{booktitle}{#1}

\DeclareFieldFormat{pages}{#1}

\renewbibmacro{in:}{}

\DeclareFieldFormat{url}{%
  \iffieldundef{doi}{\mkbibacro{URL}\addcolon\space\url{#1}}{}%
}

\title{\bf Note on BDL property of fixed points of primitive morphisms}
\author{Petr Ambrož \qquad Edita Pelantová \\[2mm]
\texttt{\{petr.ambroz,edita.pelantova\}@fjfi.cvut.cz}\\[1mm]
Faculty of Nuclear Science and Physical Engineering\\ Czech Technical University in Prague}

\date{}

\begin{document}
\maketitle
\allowdisplaybreaks
\begin{abstract} We consider an infinite word $\bu$ fixed by a primitive morphism. We show a necessary condition under which $\bu$ has a non-trivial geometric representation which is bounded distance equivalent to a lattice.  

\medskip


\medskip

\end{abstract}

\section{Introduction}

In this paper we continue our study~\cite{AmMaPe-BDL1} of geometric representations of bi-directional infinite 
words $\bu=\cdots u_{-2}u_{-1}|u_0u_{1}u_{2}\cdots$ over a finite alphabet $\A$. Here, the term 'geometric representation' 
stands for an infinite discrete set $\{x_n: n \in \mathbb{Z}\}\subset \mathbb{R}$, where $(x_n)_{n\in\Z}$ is a strictly 
increasing sequence such that if $u_n=u_m$ then $x_{n+1}-x_n=x_{m+1}-x_m$. In particular, as the alphabet $A$ is finite, 
we have only finitely many different gaps between neighbors in the sequence $(x_n)_{n\in\Z}$. 

Obviously, any lattice in $\R$, i.e., any set $\eta\Z$ with $\eta\in\R$, $\eta\neq0$, is a geometric representation of 
any infinite word $\bu$. Such representation is called trivial as every letter of $\bu$ is represented by the same gap 
(gap of the same length).

We are interested in the following problem: is there a geometric representation of $\bu$ which is not a lattice per se, 
but which is (in a way) similar to a lattice? This desired similarity can be formally expressed by the notion of bounded 
distance equivalence (to a lattice).  

\begin{de}
  We say that a set $\Lambda\subset \R$ is bounded distance equivalent to a lattice $L\subset \R$,  if there exist a constant $C$ and a bijection
  ${g}:\Lambda\to L$ such that $|x-g(x)|<C$ for all $x\in\Lambda$. 
  In short, we say that such $\Lambda$ has BDL property.
\end{de}

The notion 'bounded distance equivalent to a lattice' has been introduced when studying diffraction properties of sets modelling non-crystallographic materials (see e.g.~\cite{Baranidharan}). Recently, the question of whether the so-called Delone sets have BDL property has been considered. Several authors considered Delone sets arising via the cut-and-project method~\cite{DuOg,Fret1,haynes-etds-36} as well as Delone sets arising from substitution tilings~\cite{aliste-prieto-aihp-30,Fret1,solomon-ijm-181}. 

In~\cite{AmMaPe-BDL1} we proved two sufficient conditions for an infinite word $\bu$ to have a non-trivial geometric representation with BDL property. We showed that if $\bu$ is the so-called balanced word then it has a non-trivial BDL geometric representation. The second sufficient condition has been formulated for a class of infinite words -- for fixed points of primitive substitutions. In this paper, we prove (cf.\ Theorem~\ref{prop:BDL_repr_of_fixed_pointSufficient}) a necessary condition for words in the latter class.

\section{Preliminaries}

Let $\A$ be a finite alphabet. The set of finite words over $\A$, equipped with the operation of concatenation and the empty word $\epsilon$ as the neutral element, is a monoid, which we denote by $\A^*$.  We will also consider infinite words, namely 
bi-directional infinite words $\bu=\cdots u_{-2}u_{-1}|u_0u_1u_2\cdots\in\A^\Z$. The delimiter $|$ is important when deciding whether two bi-directional infinite words coincide.

If a word $u$ (finite or infinite) is written as $u=vwv'$ for some (possibly empty) words $v,w,v'$, then $v$ is a prefix, $v'$ a suffix and $w$ a factor of $u$. In particular, we denote by $u_{[i,j)}=u_iu_{i+1}\cdots u_{j-1}$.  The length of a finite word $w=w_1\cdots w_n$ is denoted by $|w|=n$. The number of letters $a\in\A$ occurring in the word $w$ is denoted by $|w|_a$. For a finite word $w$ over an alphabet $\A=\{a_1,\dots,a_d\}$ we define its Parikh vector $\Parikh(w)=(|w|_{a_1},\dots,|w|_{a_d})^{\mathrm{T}}$.

A mapping $\psi:\A^*\to\B^*$ is a morphism, if $\psi(wv)=\psi(w)\psi(v)$ for any $w,v\in\A^*$. The action of a morphism is naturally extended to infinite words $\bu\in\A^\Z$, by
\[
\psi(\cdots u_{-2}u_{-1}|u_0u_1u_2\cdots)=\cdots\psi(u_{-2})\psi(u_{-1})|\psi(u_0)\psi(u_1)\psi(u_2)\cdots.
\]
If $\A$ and $\B$ coincide, i.e., $\psi:\A^*\to\A^*$, and, moreover, there are letters $a, b \in \A$ and non-empty words $v,w\in\A^*$ such that $\psi(a) = aw$ and $\psi(b) = vb$ then $\psi$ is called substitution. Every substitution has at least one fixed point, namely
\[
\cdots\psi^3(v)\psi^2(v)\psi(v)vb|aw\psi(w)\psi^2(w)\psi^3(w)\cdots.
\]

To any morphism $\psi:\A^*\to\B^*$, we associate its incidence matrix $M_\psi$. Its rows and columns are indexed by $b\in\mathcal{B}$ and $a\in\mathcal{A}$, respectively.  We define $(M_\psi)_{ba}=|\psi(a)|_b$.  Given a finite word $w\in\A^*$, the Parikh vector $\Parikh(\psi(w))$ of its image $\psi(w)$ can be calculated from the Parikh vector $\Parikh(w)$ of $w$ by $\Parikh(\psi(w))=M_\psi\Parikh(w)$.

\section{Geometric representation of infinite words}

For the definition of a geometric representation of an infinite word $\bu$ we use the
Parikh vectors of its prefixes. Denote for simplicity   
\begin{equation*}
  \Parikh[n](\bu):=\left\{\begin{array}{@{}r@{}l@{\quad}l}
           \Parikh & (u_{[0,n)}) & \text{if } n\geq 0, \\[2mm]
           -\Parikh & (u_{[n,0)}) & \text{otherwise}.
  \end{array}\right.
\end{equation*}

\begin{de}\label{d:geomrep}
  Let $\bu=\cdots u_{-2}u_{-1}|u_0u_1u_2\cdots$ be an infinite word over a finite alphabet
  $\A=\{1,\ldots,d\}$ and let $\ell_1,\ldots,\ell_d$ be positive numbers.  Set
  \[
  x_n:=(\ell_1,\dots,\ell_d)\Parikh[n](\bu).
  \]
  The discrete  set $\Lambda_{\bu}:=\{x_n : n\in\Z\}$ is called a geometric representation of $\bu$
  defined by the lengths $\ell_1,\ldots,\ell_d$.  We say that the geometric representation is
  non-trivial if the set of lengths has at least two elements.
\end{de}

\begin{ex}
Geometric representation of (a part of) the infinite word $\bu=\cdots CBCBCB|CBACC\cdots$
can be found in Figure~\ref{fig:geom_rep}. The lengths $\ell_1,\ell_2,\ell_3$ correspond to letters $A$, $B$, $C$, respectively. Let us compute several points of $\Lambda_{\bu}$:
\begin{itemize}
\item 
    $u_{[0,1)}=C$ $\Rightarrow$  
    $\Parikh[1](\bu)=\Parikh(u_{[0,1)})
    =\left(\begin{smallmatrix}0\\0\\1\end{smallmatrix}\right)$
    $\Rightarrow$ $x_1=\ell_3$,
\item
    $u_{[0,2)}=CB$ $\Rightarrow$
    $\Parikh[2](\bu)=\Parikh(u_{[0,2)})
    =\left(\begin{smallmatrix}0\\1\\1\end{smallmatrix}\right)$
    $\Rightarrow$ $x_2=\ell_2 + \ell_3$,
\item
    $u_{[-3,0)}=BCB$ $\Rightarrow$
    $\Parikh[-3](\bu)=-\Parikh(u_{[-3,0)})
    =\left(\begin{smallmatrix}0\\-2\\-1\end{smallmatrix}\right)$
    $\Rightarrow$ $x_{-3}=-2\ell_2 - \ell_3$.
\end{itemize}
\end{ex}

\begin{figure}[!htp]
    \centering
    \begin{tikzpicture}[x=2cm,y=1.5cm]
		\draw [|-|] (-1.15,0) -- node [anchor = south] {C} (-0.4,0);
		\draw [-|] (-0.4,0) -- node [anchor = south] {B} (0,0);
		\draw [-|] (0,0) -- node [anchor = south] {C} (0.75,0);
		\draw [-|] (0.75,0) -- node [anchor = south] {B} (1.15,0);
        \draw [-|] (1.15,0) -- node [anchor = south] {C} (1.9,0);
        \draw [-] (1.9,0) -- node [anchor = south] {B} (2.3,0);
        \draw [-|] (2.3,0) -- node [anchor = south] {C} (3.05,0);
        \draw [-|] (3.05,0) -- node [anchor = south] {B} (3.45,0);
        \draw [-|] (3.45,0) -- node [anchor = south] {A} (4.65,0);
        \draw [-|] (4.65,0) -- node [anchor = south] {C} (5.4,0);
        \draw [-|] (5.4,0) -- node [anchor = south] {C} (6.15,0);
        \draw (2.3,0.1) -- (2.3,-0.1) node [anchor = north] {0};
        \draw [thick,decorate,decoration={brace,mirror,raise=4pt}] (3.05,0) -- node [below=6pt] {$\ell_2$} (3.44,0);
		\draw [thick,decorate,decoration={brace,mirror,raise=4pt}] (3.46,0) -- node [below=6pt] {$\ell_1$} (4.64,0);
		\draw [thick,decorate,decoration={brace,mirror,raise=4pt}] (4.66,0) -- node [below=6pt] {$\ell_3$} (5.4,0);
	\end{tikzpicture}
	\caption{Illustration of a geometric representation of an infinite word}
	\label{fig:geom_rep}
\end{figure}
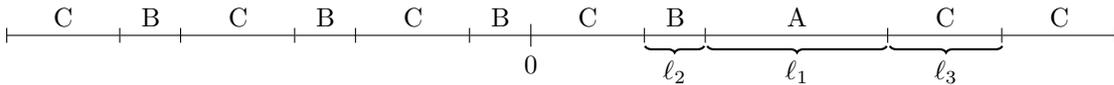

Let $\bu=\cdots u_{-2}u_{-1}|u_0u_1u_2\cdots$ be an infinite word over $\A=\{1,\ldots,d\}$, $\Lambda_{\bu}:=\{x_n : n\in\Z\}$ its geometric representation defined by the lengths $\ell_1,\ldots,\ell_d$. Let us assume that $u_m=u_n=i$ for some $m,n\in\Z$ and $i\in\A$. Then 
\begin{align}\label{eq:xn1-xn}
x_{n+1} &= (\ell_1,\ldots,\ell_d)\Parikh[n+1](\bu)
= (\ell_1,\ldots,\ell_d)
\left(\Parikh[n](\bu)+\vec{e}_i\right)=
(\ell_1,\ldots,\ell_d)\Parikh[n](\bu)+\ell_i = x_n + \ell_i,
\intertext{and} \label{eq:xm1-xm}
x_{m+1} &= (\ell_1,\ldots,\ell_d)\Parikh[m+1](\bu)
= (\ell_1,\ldots,\ell_d)
\left(\Parikh[m](\bu)+\vec{e}_i\right)=
(\ell_1,\ldots,\ell_d)\Parikh[m](\bu)+\ell_i = x_m + \ell_i,
\end{align}
where $\vec{e}_i$ is $i$-th vector of the standard basis of $\R^d$, i.e., the $d$-dimensional vector, whose only non-zero coordinate is $e_i=1$.
This shows that the set $\Lambda_{\bu}$, given by Definition~\ref{d:geomrep}, meets the condition on a geometric representation of an infinite word.
Indeed, equations~\eqref{eq:xn1-xn} and~\eqref{eq:xm1-xm} show that the sequence 
$(x_n)_{n\in \mathbb Z}$  has the property: if $u_n = u_m$ then $x_{n+1} - x_n = x_{m+1} - x_m$ for all $m,n\in\Z$.

\medskip

In~\cite{AmMaPe-BDL1} authors gave following reformulation of BDL property of a geometric
representation of an infinite word in terms of Parikh vectors of its prefixes.

\begin{lem}\label{l:vrstva}
  Let $\bu$ be a bi-directional infinite word over an alphabet $\A=\{1,\ldots,d\}$, $\eta \in
  \mathbb{R}$ and let $\ell_1, \ell_2,\dots,\ell_d$ be a list of positive numbers containing at
  least two distinct elements. Then the geometric representation of $\bu$ with lengths
  $\ell_1,\ldots,\ell_d$ is bounded distance equivalent to the lattice $\eta\Z$ if and only if for
  some constant $C$
  \[
  \Big|(\ell_1-\eta,\dots,\ell_d-\eta)\Parikh[n](\bu)\Big|<C\qquad\text{ for every } n\in\Z.
  \]
\end{lem}
\begin{proof}
   Recall that the geometric representation of $\bu$, that is, the set $\{x_n : n\in\Z\}$ is BDL to a lattice $\eta\Z$ if there exists $C > 0$ such that $|x_n- \eta n| < C$ for all $n\in\Z$. Since $x_n=(\ell_1,\ldots,\ell_d)\Parikh[n](\bu)$ and $n=(1,\ldots,1)\Parikh[n](\bu)$ we have for all $n\in\Z$
   \[
   |x_n-\eta n| = \big|(\ell_1,\ldots,\ell_d)\Parikh[n](\bu)-\eta(1,\ldots,1)\Parikh[n](\bu)\big| = \big|(\ell_1-\eta,\ldots,\ell_d-\eta)\Parikh[n](\bu)\big|.
   \]
\end{proof}

As a corollary we can derive another -- more geometric -- reformulation.  Let us recall that the
euclidean distance between a point $\vec{y}\in \R^d$ and a hyperplane $H\subset\R^d$ containing
$\vec{0}$ with unit normal vector $\vec{a}\neq\vec{0}$ is equal to
\[
\varrho(H, \vec{y})={|\vec{a}^{\,\mathrm{T}}\vec{y}|}.
\]

\begin{prop}\label{nadrovinaH}
  Let $\bu$ be a bi-directional infinite word over an alphabet $\A=\{1,\ldots,d\}$. Then $\bu$ has a
  non-trivial geometric representation which is BDL if and only if there is a hyperplane
  $H\subset\R^d$ and a constant $D>0$ such that
  \[
  \varrho(H,\Parikh[n](\bu))<D\qquad\text{for all $n\in\Z$.}
  \]
\end{prop}

\begin{proof}
  $\Rightarrow$: By Lemma~\ref{l:vrstva}, there exists a non-zero vector
  $\vec{f}:=(\ell_1-\eta,\ell_2-\eta,\ldots,\ell_d-\eta)^{\mathrm{T}}$ and a constant $C$ such that
  the distance between $\Parikh[n](\bu)$ and the hyperplane
  \[
  H := \big\{ \vec{x}\in\R^d : \vec{f}^{\:\mathrm{T}}\vec{x} = 0 \big\}
  \]  is bounded by $D:=C /\|\vec{f}\|$.  
    
  $\Leftarrow$: Let $\vec{h} = (h_1,h_2, \ldots, h_d)^{\mathrm{T}}$ be a normal vector of the
  hyperplane $H$ such that $\|\vec{h}\| = 1$. Then $\vec{h}$ has at least two distinct
  components. Indeed, if $\vec{h}$ was equal to $\frac{1}{\sqrt{d}}(1,1,\ldots, 1)^{\mathrm{T}}$,
  the sequence $\varrho(H,\Parikh[n](\bu)) = \frac{1}{\sqrt{d}} n$ would not be bounded.  To define
  a non-trivial geometrical representation of $\bu$, we find $\eta\in\R$ such that all components of
  the vector $ \vec{h} + \eta(1,1,\ldots, 1)^{\mathrm{T}}$ are positive and we assign to each letter $a
  \in \A$ the length $\ell_a = h_a+\eta$.  Obviously, at least two lengths differs. Our choice of
  lengths guarantees that
  \[
  \big|(\ell_1-\eta,\dots,\ell_d-\eta)\Parikh[n](\bu)\big| =
  \big|\vec{h}^{\,\mathrm{T}}\Parikh[n](\bu)\big| < D
  \qquad\text{ for every } n\in\Z.
  \]
  By Lemma~\ref{l:vrstva} the geometrical representation of $\bu$ is bounded distance equivalent to
  the lattice $\eta \Z$.
\end{proof}

\section{The BDL Property of Fixed Points of Substitutions}

Our aim here is to describe fixed points of substitutions having non-trivial geometrical representation which is bounded distance equivalent to a lattice.  A sufficient condition for this property is given by following proposition taken from~\cite{AmMaPe-BDL1}.

\begin{prop}\label{prop:BDL_repr_of_fixed_pointSufficient}
  Let $\varphi$ be a substitution over an alphabet $\A$ and suppose that its incidence matrix has at
  least one eigenvalue in modulus less than 1.  Let $\bu$ be a bidirectional fixed point of
  $\varphi$. Then there exists a non-trivial geometric representation of $\bu$ which is bounded
  distance equivalent to a lattice.
\end{prop}

A counterpart of the previous result is the following necessary condition.  

\begin{thm}
\label{prop:BDL_repr_of_fixed_pointNecessary}
  Let $\bu$ be a fixed point of a primitive substitution $\varphi$ over an alphabet $\A$.  If there
  exists a non-trivial geometric representation of $\bu$ which is bounded distance equivalent to a
  lattice, then at least one eigenvalue in modulus is less or equal to 1.
\end{thm}

\begin{proof}  
  By Proposition \ref{nadrovinaH} there exist a constant $D$ and a hyperplane $H$ such that
  $\varrho(H,\Parikh[n](\bu))<D$ for all $n\in\Z$. If we denote $\vec{f}$ the unit normal vector of
  $H$, then equivalently $|\vec{f}^{\:\mathrm{T}}\Parikh[n](\bu)| < D$ for all $n\in\Z$.

  Let $w$ be a factor of $\bu$.
  As $\bu$ is a fixed point of a primitive substitution, $w$
  occurs in $\bu$ infinitely many times. Hence 
  there are $n, m \in \N$, $m<n$ such that we can write $u_{[0,n)} = u_{[0,m)}w$. The Parikh
  vector $\Parikh(w)$ of the factor $w$ satisfies $\Parikh(w) = \Parikh[n](\bu) - \Parikh[m](\bu)$,
  and consequently
  $|\vec{f}^{\:\mathrm{T}}\Parikh(w)| \leq |\vec{f}^{\:\mathrm{T}}\Parikh[n](\bu)| +
  |\vec{f}^{\:\mathrm{T}}\Parikh[m](\bu)| < 2D$.
  Therefore
  \begin{equation}\label{kazdyFaktor}
    \big|\vec{f}^{\:\mathrm{T}}\Parikh(w)\big| = \big|\Parikh(w)^{\mathrm{T}} \vec{f}\big| < 2D
    \quad \text{for each factor $w \in \mathcal{L}(\bu)$}.    
  \end{equation}
  The language of the primitive substitution $\varphi$ over $\A$ contains all letters of the
  alphabet $\A$ and all iterations of its images, i.e., $\varphi^n(a) \in \mathcal{L}(\bu)$ for each
  $a \in \A$ and $n \in \N$. As $\Parikh(\varphi^n(a)) = M_{\varphi}^n\Parikh(a)$, the equation
  \eqref{kazdyFaktor} implies that all coordinates of the vector
  $(M^n_\varphi)^{\mathrm{T}}\!\vec{f}$ are bounded in modulus by $2D$.  In other words, the
  vectors $(M^n_\varphi)^{\mathrm{T}}\!\vec{f}$ lie in a cube $\mathcal{C}$ centered at the origin.

  Assume, contrary to what we want to prove, that the matrix $M_\varphi$ is expansive, i.e., all its
  eigenvalues are in modulus strictly bigger than 1. This implies that
  $A:=(M_\varphi^{-1})^{\mathrm{T}}$ is contracting, i.e., all eigenvalues of $A$ are strictly
  smaller than $1$. Let us choose $\beta$ satisfying
  \[
  1 > \beta > \max\big\{|\lambda|:\lambda \text{ an eigenvalue  of } A \big\}.
  \]
  A well-know result (see for example~\cite{isaacson}, Theorem 3) says that there exists a norm
  $\|\cdot\|_{\beta}$ of $\R^d$ such that $\|A \vec{x}\|_{\beta}\leq \beta \|\vec{x}\|_{\beta}$ for each
  $\vec{x} \in \R^d$. By induction on $n$ we deduce that $\|A^n
  \vec{x}\|_{\beta}\leq\beta^n\|\vec{x}\|_{\beta}$. As all norms on a finite dimensional vector
  space are equivalent, the cube $\mathcal{C}$ is a bounded set also in the norm $\|\cdot\|_{\beta}$,
  i.e., there exists a constant $K$ such that $\|\vec{x} \|_{\beta}\leq K$ for each
  $\vec{x}\in\mathcal{C}$.

  Let us consider the sequence $x_n := (M^n_\varphi)^{\mathrm{T}}\!\vec{f}$ of vectors belonging to
  the cube $\mathcal{C}$.  Then
  \begin{equation}\label{eq:Anxn}
  \|A^n x_n\|_{\beta}\leq \beta^n \|x_n\|_{\beta} \leq K \beta^n \qquad\text{and}\qquad
  A^n x_n = \big((M_\varphi^{-1})^{\mathrm{T}}\big)^n(M^n_\varphi)^{\mathrm{T}}\vec{f} = \vec{f} 
  \quad \text{for each } n \in \N.
  \end{equation}
  Since $0<\beta < 1$, $\lim\limits_{n\to \infty}K\beta^n = 0$.  Formulae~\eqref{eq:Anxn} imply
  $\vec{f} = \vec{0}$; a contradiction as $\vec{f}$ is a normal vector of the hyperplane $H$.
\end{proof}

Let us show that the necessary condition for the BDL Property we have demonstrated in the previous
proposition is not sufficient.

\begin{ex}   
  Let us consider  the primitive substitution  $\psi: \{A,B,C\}^* \to \{A,B,C\}^*$ given by 
  \[
  A\mapsto BBBCCC, \quad B\mapsto BACCB, \quad \text{and} \quad C\mapsto ABBBC.
  \]  
  We will show that the fixed point 
  $\bu=\cdots ABBBCABBBCBABBA|BACCBBBBCCCABBC\cdots$
  of $\psi$ has no
  non-trivial geometrical representation, although one eigenvalue of 
  $M_{\psi}$ is in modulus $1$.

  The matrix of $\psi$  and its eigenvalues  are  
  \[
  M_\psi = \begin{pmatrix} 0&1&1 \\ 3&2&3  \\ 3&2&1  \\ \end{pmatrix}
    \quad \text{and}\quad \lambda_1=2+\sqrt{10},\ \lambda_2=2-\sqrt{10},\ \lambda_3=-1. 
  \]
  The Parikh vectors of the factors $\psi^n(A)$, $\psi^n(B)$, and $\psi^n(C)$ equal to the first, second, and third column of the matrix $ M_{\psi}^n$, respectively.

  Let us assume that the fixed point $\bu$ of $\psi$ has the BDL Property. Let $\vec{f}$ be a normal vector of a hyperplane $H$ from Proposition~\ref{nadrovinaH}, then the coordinates of the vectors $(M^n_\psi)^{\mathrm{T}}\!\vec{f}$ are be bounded (cf.~proof of Proposition~\ref{prop:BDL_repr_of_fixed_pointNecessary}). Hence the product $\vec{x}^{\:\mathrm{T}}(M^n_\psi)^{\mathrm{T}}\vec{f}$ is bounded for each $\vec{x}\in \R^3$. In particular, if $\vec{x}$ is an eigenvector of $M_\psi$ to the eigenvalue $\lambda$, then $\vec{x}^{\:\mathrm{T}}(M^n_\psi)^{\mathrm{T}}\vec{f} =\lambda^n\vec{x}^{\:\mathrm{T}}\vec{f}$ is bounded, and thus $\vec{x}^{\:\mathrm{T}}\vec{f} = 0$ for each eigenvector $\vec{x}$ of $M_\psi$ corresponding to an eigenvalue in modulus bigger than $1$.  The eigenvectors of $M_\psi$ corresponding to $\lambda_1$ and $\lambda_2$ are $(1,3, \sqrt{10}-1)^{\mathrm{T}}$ and $(1,3,-\sqrt{10}-1)^{\mathrm{T}}$, respectively. Therefore the only candidate (up to a scalar multiple) for $\vec{f}$ is $\vec{f} = (3, -1,0)^{\mathrm{T}}$.

  As $M_{\psi}$ is diagonalizable, we can write $M_{\psi}=R\,\mathrm{diag}(\lambda_1,\lambda_2,\lambda_3)R^{-1}$, where columns of the matrix $R$ are formed by eigenvectors of $M_\psi$. Therefore
  \[ 
  \big(M_\psi^n\big)^{\mathrm{T}}\vec{f} = \big(R^{-1}\big)^{\mathrm{T}}
  \mathrm{diag}(\lambda^n_1, \lambda^n_2, \lambda^n_3)R^{\mathrm{T}}\vec{f} =
  \big(R^{-1}\big)^{\mathrm{T}} \big(0,0,(-1)^{n+1}\big)^{\mathrm{T}} = \big(3(-1)^n,(-1)^{n+1}, 0\big)^{\mathrm{T}}. 
  \]
  In particular, for the Parikh vectors of the factors  $\psi^n(A)$, $\psi^n(B)$, and $\psi^n(C)$ we get 
  \begin{equation}\label{mocniny}
    \vec{f}^{\:\mathrm{T}}\Parikh(\psi^n(A)) = 3(-1)^n,
    \quad \vec{f}^{\:\mathrm{T}}\Parikh(\psi^n(B)) = (-1)^{n+1}, \quad \text{and}
    \quad \vec{f}^{\:\mathrm{T}}\Parikh(\psi^n(C)) = 0.
  \end{equation}
  The form of $\psi$ implies two simple claims:
  \begin{itemize}
  \item 
    if $wB$ is a prefix of $u_{[0,\infty)}$, then $\psi(w)BAC$ is a prefix of $u_{[0,\infty)}$,
  \item 
    if $wC$ is a prefix of $u_{[0,\infty)}$, then $\psi(w)ABBB$ is a prefix of $u_{[0,\infty)}$. 
  \end{itemize}
  Starting with the prefix $B$ and applying alternatively these claims we find out that 
  \[ 
  F_k : = \psi^{2k}(BA)\psi^{2k-1}(ABB)\psi^{2k-2}(BA)\psi^{2k-3}(ABB)\cdots\psi^2(BA) \psi(ABB)BAC
  \]
  is a prefix of $u_{[0,\infty)}$ for each $k\in\N$. As $\Parikh(uv)=\Parikh(u)+\Parikh(v)$ for any two words
  $u,v$, the Parikh vector of $F_k$ is
  \[
  \Parikh(F_k) = \sum_{i=0}^{2k}\Parikh(\psi^i(A)) + \sum_{i=0}^{2k}\Parikh(\psi^i(B)) +
  \sum_{i=1}^{k}\Parikh(\psi^{2i-1}(B)) + \Parikh(C).
  \]
  To compute $\vec{f}^{\:\mathrm{T}}\Parikh(F_k)$, we use equalities~\eqref{mocniny}
  \[
  \vec{f}^{\:\mathrm{T}}\Parikh(F_k) =  3\sum_{i=0}^{2k}(-1)^i + \sum_{i=0}^{2k}(-1)^{i+1} +\sum_{i=1}^{k}(-1)^{2i} = k+2\,.
  \]
  
  To sum up: first, we showed that $\vec{f} = (3, -1,0)^{\mathrm{T}}$ is the only candidate for a
  vector such that $|\vec{f}^{\:\mathrm{T}}\Parikh[n](\bu)|<D$ for some constant $D$. 
  Then, we found a sequence $(F_k)_{k\geq 0}$ of prefixes of $u_{[0,\infty)}$ such that
  $\lim_{k\to\infty}\vec{f}^{\:\mathrm{T}}\Parikh(F_k) = \infty$. Therefore by
  Proposition~\ref{nadrovinaH} the fixed point $\bu$ of $\psi$ admits no non-trivial geometrical
  representation.
\end{ex}

\section{Morphic images of words having BDL property}   

In our previous paper~\cite{AmMaPe-BDL1} we showed that if an infinite word $\bu$ is balanced, then its geometric representation has property BDL. Moreover, we proved that the balancedness is preserved under the image by a morphism, that is, $\psi(\bu)$ is balanced as well for any morphism $\psi:\A^*\to\B^*$. In particular, a geometric representation of $\psi(\bu)$ has property BDL.

In this Section we prove that for an arbitrary infinite word the BDL property of its geometric representation is preserved under the image by a morphism $\psi:\A^*\to\B^*$, provided that the cardinality of $\B$ is greater than or equal to the cardinality of $\A$. We also povide an example showing that this condition on the cardinalities is essential.

\begin{prop}\label{prop:BDL_of_image}
    Let $\bu\in\mathcal{A}^{\Z}$ be an infinite word with a non-trivial BDL geometric representation
    and let $\varphi:\mathcal{A}^*\rightarrow\mathcal{B}^*$ be a morphism such that
    $\#\mathcal{B}\geq\#\mathcal{A}$. Then $\varphi(\bu)$ has a non-trivial BDL geometric
    representation.
\end{prop}

\begin{proof}
    Let $d_{\mathcal{A}}:=\#\mathcal{A}$ and $d_{\mathcal{B}}:=\#\mathcal{B}$, and let
    $H_{\mathcal{A}}\subset\R^{d_{\mathcal{A}}}$ be a hyperplane and $C_{\mathcal{A}}$ a constant
    such that
    $\rho(H_{\mathcal{A}},\Parikh[n](\bu))<C_{\mathcal{A}}$.
    Thus for any $n\in\N$ we can write
    \begin{equation}\label{eq:rozklad_Pnu}
    \Parikh[n](\bu) = y_n + z_n,\qquad\text{where } y_n\in H_{\mathcal{A}}
    \text{ and }\|z_n\|<C_{\mathcal{A}}.
    \end{equation}
    Let $\boldsymbol{w}=\varphi(\bu)$, then for a prefix of $\boldsymbol{w}$ of length $m$ we can
    write $w[0,m)=\varphi(u[0,n))v$ for some $n\in\N$ and $v$ a prefix of $\varphi(a)$,
        $a\in\mathcal{A}$. Thus
    \[
    \Parikh[m](\boldsymbol{w})=M_{\varphi}\Parikh[n](\bu)+\Parikh(v) =
    M_{\varphi}y_n + M_{\varphi}z_n + \Parikh(v).
    \]
    Let $H_{\mathcal{B}}\subset\R^{d_{\mathcal{B}}}$ be a hyperplane containing the subspace
    $M_{\varphi}H_{\mathcal{A}}$. Such a hyperplane surely exists since $\dim
    M_{\varphi}H_{\mathcal{A}}\leq\dim H_{\mathcal{A}}= d_{\mathcal{A}}-1$ and, by assumption,
    $d_{\mathcal{A}}\leq d_{\mathcal{B}}$. Finally, as $M_{\varphi}y_n\in H_{\mathcal{B}}$ the
    distance between $\Parikh[m](\boldsymbol{w})$ and $H_{\mathcal{B}}$ is
    \[
    \varrho\bigl(H_{\mathcal{B}},\Parikh[m](\boldsymbol{w})\big) =
    \inf\big\{\|\Parikh[m](\boldsymbol{w})-x\|:x\in H_{\mathcal{B}}\big\}
    \leq \|M_{\varphi}z_n+\Parikh(v)\|<C_{\mathcal{B}},
    \]
    for some constant $C_{\mathcal{B}}$. The last inequality follows from the fact that 
    $\|z_n\|<C_{\mathcal{A}}$.
\end{proof}

We will show that the condition on the size of the alphabets in
Proposition~\ref{prop:BDL_of_image} is necessary.

\begin{ex}
  Let $\bu$ be a fixed point of substitution $\psi:\{A,B,C\}^*\to\{A,B,C\}^*$ given by
  \[
  A \mapsto A^2B^3C^4,\quad B \mapsto A^2B,\quad C\mapsto A^2.
  \]
  The matrix of $\psi$ and its eigenvalues are
  \[
  M_\psi = \begin{pmatrix} 2 & 2 & 2 \\ 3 & 1 & 0  \\ 4 & 0 & 0  \\ \end{pmatrix}
  \quad \text{and}\quad
  \lambda_1\approx5.0593,\ \lambda_2\approx-2.6549,\ \lambda_3\approx0.5956.
  \]
  By Proposition~\ref{prop:BDL_repr_of_fixed_pointSufficient} the fixed point $\bu$ has a BDL
  geometric representation.

  Let $\varphi:\{A,B,C\}^*\to\{A,B\}^*$ be the morphism given by
  $\varphi(A)=A$, $\varphi(B)=B$, $\varphi(C)=\epsilon$. We show by contradiction that
  $\varphi(\bu)$ does not admit a BDL geometrical representation.

  Let us assume that $\varphi(\bu)$ has a BDL geometric representation. Then (by
  Proposition~\ref{nadrovinaH}) there is a non-zero vector
  $\vec{f}=\left(\begin{smallmatrix}f_1\\f_2\end{smallmatrix}\right)$ and a constant $C$ such that
  $\vec{f}^{\:\mathrm{T}}\Parikh[n](\varphi(\bu))$ is bounded by $C$ for every $n\in\Z$. Let us
  consider $w_n$, 
  factors of $\varphi(\bu)$, given by $w_n=\varphi(\psi^n(A))$ for all
  $n\in\N$. Parikh vectors of these factors are
  \[
  \Parikh(w_n) = M_{\varphi} M_{\psi}^n \begin{pmatrix}1\\0\\0\end{pmatrix}.
  \]
  The matrix $M_{\psi}$ is diagonalizable, thus we can write $M_{\psi} = R D R^{-1}$, 
  where $D = \mathrm{diag}(\lambda_1,\lambda_2,\lambda_3)$.
  Let us denote $r_{ij}:=[R]_{ij}$ and $s_{ij}:=[R^{-1}]_{ij}$ and note that for $M_{\psi}$ all these
  elements of matrices $R$ and $R^{-1}$ are non-zero.

  We have 
  \begin{align*}
  \vec{f}^{\:\mathrm{T}}\Parikh(w_n) &= 
  (f_1,f_2)
  \begin{pmatrix}1&0&0\\0&1&0\end{pmatrix} R D^n R^{-1} \begin{pmatrix}1\\0\\0\end{pmatrix} \\
      &= 
      (f_1r_{11}+f_2r_{21}, f_1r_{12}+f_2r_{22}, f_1r_{13}+f_2r_{23})
      D^n \begin{pmatrix}s_{11}\\s_{21}\\s_{31}\end{pmatrix} \\
      &= (f_1r_{11}+f_2r_{21})s_{11}\lambda_1^n + (f_1r_{12}+f_2r_{22})s_{21}\lambda_2^n + (f_1r_{13}+f_2r_{23})s_{31}\lambda_3^n. 
  \end{align*}
  Boundedness of $\vec{f}^{\:\mathrm{T}}\Parikh(w_n)$ implies (besides other things) that
  $f_1r_{11}+f_2r_{21}=0$ and $f_1r_{12}+f_2r_{22}=0$. As the coefficients
  $r_{11},r_{21},r_{12},r_{22}$ are non-zero we get $f_1=f_2=0$, a contradiction.
\end{ex}

\section{Comments}  

In this paper we considered problem when a fixed point of a primitive substitution $\psi$ has a non-trivial geometric representation with BDL property. We showed that the existence of such representation implies that at least one eigenvalue of the incidence matrix of $\psi$ is in modulus less than or equal to one. Several examples we investigated support the conjecture that such geometric representation with BDL property exists if and only if at least one eigenvalue of $M_{\psi}$ is in modulus striclty smaller than 1.

On the other hand, we did not at all considered problem when an infinite word over an alphabet $\A$ has the so-called faithful BDL representation, i.e. when the lengths $\ell_a$ corresponding to letters $a\in\A$ are mutually different. In~\cite{AmMaPe-BDL1} we proved that all balanced words have faithful geometric representation. The existence of faithful representation for other classes of infinite words
is an open question.    
    

\section*{Acknowledgments}

This work was supported by the project CZ.02.1.01/0.0/0.0/16\_019/0000778 from European Regional
Development Fund.


\printbibliography


\end{document}